\documentclass[11pt]{amsart}
\usepackage{amssymb, amsmath, amsthm,ulem}
\usepackage[latin1]{inputenc}
\usepackage{graphicx}
\usepackage[hidelinks]{hyperref}
\usepackage{color}
\usepackage{epstopdf}
\usepackage{cancel}
\usepackage{tikz}
 \usepackage{enumitem}

\usepackage[a4paper, centering]{geometry}

\usepackage{scalerel,stackengine}

\usepackage{mathtools}

\newtheorem{theorem}{Theorem}
\newtheorem{proposition}[theorem]{Proposition}

\newtheorem{lemma}[theorem]{Lemma}

\theoremstyle{definition}
\newtheorem{definition}[theorem]{Definition}
\theoremstyle{remark}

\definecolor{verde}{RGB}{20,150,100}
\definecolor{purple}{RGB}{200,30,200}
\definecolor{tortora}{RGB}{145, 135, 125}

\def\HH{\mathcal{H}}

\stackMath
\newcommand\reallywidecheck[1]{%
\savestack{\tmpbox}{\stretchto{%
  \scaleto{%
    \scalerel*[\widthof{\ensuremath{#1}}]{\kern-.6pt\bigwedge\kern-.6pt}%
    {\rule[-\textheight/2]{1ex}{\textheight}}
  }{\textheight}%
}{0.5ex}}%
\stackon[1pt]{#1}{\scalebox{-1}{\tmpbox}}%
}

\def\T{\tau}

\def\R{\mathbb{R}}

\newcommand{\Om}{\Omega}

\def \e{\varepsilon}

\newcommand{\vphi}{\varphi}

\begin{document}
\title[]{ 
Proof of the local version of the P\'olya--Szeg\"o conjecture 
for the torsional rigidity of polygons 
}

\bigskip\bigskip

\vfill\eject 

\author[]{Beniamin Bogosel, Dorin Bucur, Ilaria Fragal\`a}

\thanks{}

\address[Beniamin Bogosel]{
	Faculty of Exact Sciences, Aurel Vlaicu University of Arad, 2 Elena Dr\u agoi Street, Arad, Romania}
\email {beniamin.bogosel@uav.ro}

\address[Dorin Bucur]{
Universit\'e  Savoie Mont Blanc, Laboratoire de Math\'ematiques CNRS UMR 5127 \\
  Campus Scientifique \\
73376 Le-Bourget-Du-Lac (France)
}
\email{dorin.bucur@univ-savoie.fr}

\address[Ilaria Fragal\`a]{
Dipartimento di Matematica \\ Politecnico  di Milano \\
Piazza Leonardo da Vinci, 32 \\
20133 Milano (Italy)
}
\email{ilaria.fragala@polimi.it}

\date{\today}


\begin{abstract}  
We prove the local version of  the 
P\'olya--Szeg\"o conjecture 
for the torsional rigidity of polygons: 
 for every \(n\geq5\), the regular \(n\)-gon is a strict local maximizer of torsional rigidity among convex \(n\)-gons of prescribed area.   Our proof is entirely analytic. It builds on a locally stable proportional triangular covering inspired by Solynin and Zalgaller and an associated weighted Voronoi-type partition, together with a quantitative   asymptotic analysis of the   loss    produced by truncating the  overlapping triangles to the partition cells. The result is then obtained  by establishing two key ingredients: the optimality of isosceles triangles for mixed torsional rigidity at fixed area and vertex angle, and a strict concavity property of the mixed torsional rigidity of isosceles triangles.   
 \end{abstract} 

\maketitle

\section{Introduction and statement of the result}
\label{s:main}

A classical question in shape optimization is whether symmetry of optimal
domains persists in the discrete setting: for several isoperimetric-type
problems in which balls are optimal, the corresponding question among polygons
with a prescribed number of sides is whether the regular $n$-gon is optimal.

While for the perimeter the regular $n$-gon is classically known to be optimal
among polygons with a prescribed number of sides and fixed area
(see, e.g., \cite{BuZa}), the analogous statement for other fundamental shape
functionals, including the first Dirichlet eigenvalue, the torsional rigidity,
and the logarithmic capacity, is a long-standing conjecture formulated by
P\'olya and  Szeg\"o more than fifty years ago \cite{PS}. They also observed that,
for $n=3$ and $4$, the result follows immediately from Steiner symmetrization,
whereas for $n\geq5$ this argument breaks down, since symmetrization may increase
the number of sides.

For logarithmic capacity, the conjecture was proved in 2004 by Solynin and
Zalgaller \cite{SZ}, by means of a geometric argument based on a suitable
triangular covering. Let us also mention that the analogous result for the
Cheeger constant was proved by the second and third authors in \cite{BF16},
using shape derivatives. For the first Dirichlet eigenvalue and the torsional
rigidity, the problem, in its full generality, is currently  open.

Let us recall that, given a convex polygon $D$, its torsional rigidity
$\T(D)$ is given by
$$
\T(D):=-\inf_{u\in H^1_0(D)}
\int_D\bigl(|\nabla u|^2-2u\bigr)\,dx
=\int_D u_D\,,
$$
where $u_D$ is the torsion function of $D$, namely the unique solution in
$H^1_0(D)$ of
$$
-\Delta u_D=1\qquad\text{in }D.
$$

The P\'olya--Szeg\"o conjecture asserts that, among $n$-gons of prescribed
area, the regular $n$-gon has maximal torsional rigidity. The purpose of the
present paper is to establish the conjecture locally around the regular polygon.

Before stating our result, we  briefly review the state of the art.
Let us mention that, in \cite{BFL12} and \cite{FrGaLa}, the regular polygon
was shown to optimize the torsional rigidity, respectively under a constraint
different from prescribed area, and under an area constraint within a
restricted class of convex polygons satisfying a quantitative asymmetry
condition.
A  step forward towards the solution of the P\'olya--Szeg\"o conjecture  was made in the works \cite{BoBu24,BoBu26} by the first and
second authors. They showed
that, for the first Dirichlet eigenvalue, for a given $n$ the full conjecture can be reduced to finitely many, albeit very demanding,    
certified numerical computations. By computing the shape Hessian matrix, they  established the local optimality of the
regular $n$-gon for $n=5,6$,  based on a rigorous  certification of the numerical approximations of its eigenvalues.   
The numerical approximation techniques used in \cite{BoBu26} restricted  this procedure to $n=5,6$.

In \cite{GHLZ26}, inspired by \cite{DGP26},
Gui, Hu, Li and Zhang proved that the torsional rigidity
of regular polygons with fixed area is strictly increasing
with the number of sides, 
a necessary condition for the validity
of the P\'olya--Szeg\"o conjecture.
Shortly afterwards, in \cite{GHL26}, Gui, Hu and Li,
using a triangular decomposition in the spirit of \cite{SZ},
proved that the regular polygon has maximal torsional rigidity
among tangential $n$-gons with prescribed area.
Finally, 
in \cite{Bogosel26},
the first author established strict local maximality
of regular $n$-gons for torsional rigidity
in the range $5\le n\le25$,
through a computer-assisted proof based on refined error estimates
for the Hessian approximation.

In contrast with the existing computer-assisted local results, our approach is entirely analytic and applies to every $n\geq 5$.  
 Our main result reads as follows:  

\begin{theorem}\label{t:main}
Let $n\geq5$. If $D^*$ is a regular $n$-gon, then every convex $n$-gon $D$ with
$|D|=|D^*|$ and vertices sufficiently close to those of $D^*$ satisfies  
$$
\T (D)\leq\T (D^*),
$$
with strict inequality unless $D$ is congruent to $D^*$.
\end{theorem}
 
 Our proof starts from the triangular method introduced by Solynin and Zalgaller in \cite{SZ} and recently revisited in \cite{GHL26}.  Roughly speaking, the underlying idea is the following. Assume that a polygon can be partitioned into a family of triangles, each having one side of the polygon as its base and area proportional to the interior angle opposite to that base. The torsional rigidity of the polygon can then be bounded from above by the sum of the {\it mixed} torsional rigidities of these triangles, where {\it mixed} means that homogeneous Dirichlet boundary conditions are imposed only on the base, while homogeneous Neumann boundary conditions are prescribed on the remaining part of the boundary. The problem is thus reduced to optimizing the collective behaviour of the mixed torsional rigidities of these triangles, under a constraint on the sum of their  apex angles.

In general, however, such a triangular partition does not exist. In some notable cases, such as the class of  tangential polygons,  the problem can be approached  by using, in place of a proportional partition, 
  some natural partition adapted to the geometry of the domain. For arbitrary polygons,  the strategy adopted by Solynin and Zalgaller is  to construct a suitable family of triangles satisfying the required proportionality condition, but forming a covering rather than a partition, so that overlaps between different triangles may occur. When passing from logarithmic capacity, which involves an exterior boundary value problem, to torsional rigidity, which instead involves an interior one,  the analysis becomes considerably more delicate, as evidenced  by the fact that the problem remains open more than twenty years after \cite{SZ}. The difficulty is twofold: on the one hand, the overlaps among the triangles of the covering constitute a substantial obstacle to comparing the torsional rigidity of the polygon with the mixed torsional rigidities of the individual triangles; on the other hand, optimizing the mixed torsional rigidity of a triangle under constraints on its area and on the angle at the vertex opposite to the Dirichlet side is itself a challenging problem.  
To overcome these difficulties, we combine several ingredients:
\begin{itemize}[topsep=6pt,partopsep=0pt,itemsep=5pt] 
\item{} 
We prove that, for polygons sufficiently close to the regular one, the proportional triangular covering is uniquely determined and depends continuously on the vertices, and we associate with it a suitable weighted Voronoi-type partition of the polygon whose cells are contained in the corresponding covering triangles  (see Section \ref{s:local-covering}). 
\item{}  We establish a truncation estimate for the mixed torsional rigidity,
quantifying the error produced by replacing the overlapping covering triangles by the cells of the polygonal partition,
together with an asymptotic analysis of such error near the regular polygon  (see Section \ref{s:truncation}). 
 During this process,  we rely crucially on  some information about the geometry of the level sets of the mixed torsion function recently established by the second and third author in 
\cite{BF26}.

\item{}  We prove the optimality of the isosceles triangle for the mixed torsional rigidity under fixed area and vertex angle, by means of a reflection argument reminiscent of the one used in \cite{FV19} (see Section \ref{s:isosceles}).
 Here we use  as a key tool a monotonicity property of the mixed torsion function proved by Li and Yao in  \cite{LY}.   
  
\item{} We obtain a concavity property for a suitable rescaling of the mixed torsional rigidity of isosceles triangles (see Section \ref{s:local-concavity}), close in spirit to \cite[Theorem 3.1]{GHL26}: although we do not invoke the spectral representation of the torsion function employed there, relying instead on a direct variational argument, the paper \cite{GHL26} gave us a decisive impulse to pursue and develop our earlier attempts in this direction.
\end{itemize} 
 
 The proof is then completed by combining all the above ingredients (see Section~\ref{s:combining}):  using the polygonal partition and
the quantitative truncation estimate, together with the  asymptotic analysis of the shrinking truncated regions 
near the regular polygon,
we reduce the problem to an estimate
involving the mixed torsional rigidities of the covering triangles; next, 
the optimality of the isosceles triangle allows us to replace each
of these terms by the corresponding isosceles one; finally, the concavity
property yields the required bound on their sum, using the fact that the
vertex angles of the triangular covering add up to $2\pi$.

For the reader's convenience, in Section \ref{sec:proof} we present
the summary of the proof  of Theorem \ref{t:main}, postponing the proofs of all the intermediate results  to the
subsequent sections. 

\section{Summary of the proof of Theorem \ref{t:main}}\label{sec:proof} 

\subsection{The proportional triangular covering and the associated polygonal partition}\label{s:local-covering}

Let $D$ be a convex $n$-gon having vertices
$A_1,\ldots,A_n$,  enumerated in the counterclockwise direction. Below, all indices are understood modulo $n$; in particular,
$A_{n+1}=A_1$ and $A_0=A_n$, and similarly for
$T_i$, $q_i$, $\ell_i$, $x_i$, and $y_i$.

We recall the following  
definitions from the seminal paper \cite{SZ}: 

\begin{definition} \label{d:system} 
A family of triangles
$\{T_i\}_{i=1}^n$ is called an {\it admissible system} for $D$ if, for every 
$i=1,\ldots,n$, 
$T_i \cap D \neq \emptyset$, $T_i$ has base
$[A_i,A_{i+1}]$, and $T_i$ and $T_{i+1}$ share a boundary
segment which is an entire side of at least one of these
triangles.

An admissible system $\{T_i\}_{i=1}^n$
is called {\it proportional} if, denoting by  $\alpha_i$  the angle of $T_i$ opposite to the base
$[A_i,A_{i+1}]$,   the ratio
$R:=\frac{\alpha_i}{| T_i|}$
is independent of $i$. \end{definition} 
 
 \begin{definition}\label{d:weights}   Let $\{T_i\}_{i=1}^n$ be an admissible proportional system according to the above definition. For every $i = 1, \dots, n$, 
 denote by $\ell _i$ the ray emanating from $A_i$ such that  the triangle $T_i$ with base $[A_i,A_{i+1}]$  has sides on $\ell _i$ and $\ell _{i +1}$.  On these rays, we 
define a family of positive weights $\{p_i\}_{i=1}^n$ 
as follows.  
 We take $p_0 = 1$, and then for every $i = 1, \ldots, n$ we set 
\begin{equation}\label{f:recursive} 
p_i
= p_{i-1}\frac{d(z,q_{i-1})}{d(z,q_i)}  \qquad \forall z \in \ell _i \setminus  \{A_i \}  \,,
\end{equation} 
where  $d(\cdot ,q_i)$ denotes the distance from the straight line $q_i$
containing the side
$[A_i,A_{i+1}]$ (noting that the quotient on the right-hand side is independent of the
choice of $z$ in  $ \ell _i \setminus \{A_i \}$). 

  \end{definition} 
  
It is shown in \cite[Lemma 6 and its proof]{SZ}  that: 
  \begin{itemize}
  \item[(a)]  There exists an admissible proportional system which covers $D$, i.e., such that 
  \begin{equation}\label{f:covering} 
  \bigcup_{i = 1} ^n T _ i \supseteq D\,.
  \end{equation} 
  \item[(b)] An admissible proportional system covers $D$ provided the   last weight in Definition \ref{d:weights} satisfies 
  \begin{equation}\label{f:pN} 
  p _n = 1  \,. 
  \end{equation}  
  
  \end{itemize}

The approach used in \cite{SZ} to obtain the above results is global, and can also be extended to non--convex polygons. 
However, it does not imply the uniqueness of the admissible proportional system satisfying \eqref{f:covering}, nor its stability under perturbations of the polygon. Since 
such stability is crucial for our purposes, and since we work only near the regular polygon, we give a different proof of statement (a) above, which enables us to complement it
with uniqueness of the admissible proportional covering,  and continuous dependence on the vertices of the polygon. Our proof still exploits statement (b), but relies on a local argument,
namely on the implicit function theorem. A further novelty is that we derive from the proportional covering  
a partition of $D$ into weighted Voronoi-type cells,  which will play a crucial role in estimating its torsional rigidity. 

\smallskip

After a rigid motion and a scaling,  we may work with the regular $n$-gon $D^*$ 
whose vertex tuple   is 
 $a^*=(A_1^*,\ldots,A_{n}^*)$, where 
$$
A_i^*= \big (\cos \big (\frac{2\pi i }{n} \big  ),\ \sin \big (\frac{2\pi i}{n}  \big ) \big ), \qquad i=1,\ldots,n\,.
$$ 
 \smallskip
We prove (see Section \ref{s:proof-cover}):  
 
\begin{theorem}\label{t:cover} Let $n \geq 5$. 
For  any convex polygon $D$ whose vertex tuple $a=(A_1,\ldots,A_{n})$ is sufficiently close
to $a^*=(A_1^*,\ldots,A_{n}^*)$, 
there exists a unique admissible proportional system  $\{ T _i \} _{ i = 1 } ^n$, sufficiently close to the fundamental system of $D^*$,
such that the corresponding system of weights  $\{ p _ i \} _{i = 1} ^ n$  introduced in Definition \ref{d:weights} satisfies $p _n =1$, and hence $\cup _{ i = 1 } ^n T_i \supseteq D$. 
 
 Moreover, as $a \to a ^*$, this covering converges to the fundamental covering of $D ^*$ by isosceles triangles, namely, for every $i = 1, \dots, n$, the apex of $T_i$ converges to $0$. 
 
 Finally,  the Voronoi-type cells  
 \begin{equation}\label{f:cover-cells}
 Q_i:=\{x\in D:p_i d(x,q_i)\le p_j d(x,q_j)\ \forall j\}
\end{equation}
      are convex polygons satisfying  $Q_i\subseteq T_i$, containing the whole base
      $[A_i,A_{i+1}]$ of $T_i$, and partitioning $D$. 
\end{theorem} 

\medskip 

\subsection{Mixed torsional rigidity under convex truncation and error asymptotics close to the regular polygon}\label{s:truncation} 

Given a polygon $P$ and a fixed side $\Gamma $ of $P$, 
we define its {\it mixed torsional rigidity} (with Dirichlet side $\Gamma$) by 
\begin{equation}\label{f:tmixvar}
\T_{\rm mix}(P):= - \!\!\!\inf_{u\in H^1_\Gamma(P)}\int_P\big(|\nabla u|^2-2u\big)\,dx
=\sup_{u\in H^1_\Gamma(P)\setminus\{0\}}
\frac{\left(\displaystyle\int_P u\,dx\right)^2}{\displaystyle\int_P|\nabla u|^2\,dx}\,, 
\end{equation} 
where $H^1_\Gamma(P)$ is the space of functions in $H^1(P)$ with zero trace on $\Gamma$. 
Equivalently, 
\begin{equation}\label{f:tmix}
\T_{\rm mix}(P):=\int_P |\nabla w_P|^2=\int_P  w_P\,,
\end{equation}
where $w _P$ is the 
{\it mixed torsion function} of $P$, namely the unique solution to the boundary value problem
\begin{equation}\label{f:mixedtorsion}
\begin{cases} - \Delta w = 1 & \text{ in  } P  \\
w =0 & \text{ on }\Gamma,\\
\partial_\nu w =0 & \text{ on }\partial P \setminus \Gamma \,.
\end{cases}
\end{equation}

 In order to compare the mixed torsional rigidity (with Dirichlet side $[A_i, A_{i+1}]$) of each of the 
 covering triangles $T_i$ given by Theorem \ref{t:cover} with that of the corresponding polygonal
cell $Q_i \subseteq T_i$ given by  \eqref{f:cover-cells}, we shall prove   the following result:

 \begin{theorem}
\label{t:truncation}
Let $T$ be an acute triangle with base $\Gamma$, and let
$Q\subsetneq T$ be a convex polygon containing the whole
segment $\Gamma$ in its boundary.
Denote by $w$ and $z$ the mixed torsion functions of $T$
and $Q$, respectively, with homogeneous Dirichlet condition
on $\Gamma$ and homogeneous Neumann condition on the
remaining boundary. Set
$$
R:=T\setminus Q,
\qquad
\Sigma:=\partial Q\cap\operatorname{int}T.
$$
Then
\begin{equation}\label{r:truncation-bounds}
\T_{\rm mix}(Q) \leq \T_{\rm mix}(T) -\int_R w\,dx -|R|\inf_\Sigma z \,. \end{equation}
\end{theorem} 

In turn, 
close to the regular polygon,  in order to control the error contribution arising from  the terms on 
 the right--hand side of \eqref{r:truncation-bounds}  
(evaluated at each  triangle of the covering), 
we prove 
 the following asymptotic estimate:

\begin{theorem}
\label{t:strategy-asymptotics} Let $n\geq5$. For every convex polygon $D$ whose vertex tuple
$a=(A_1,\ldots,A_n)$ is sufficiently close to
$a^*=(A_1^*,\ldots,A_n^*)$, let $\{T_i\}_{i=1}^n$ and
$\{Q_i\}_{i=1}^n$ be as in Theorem~\ref{t:cover}. For every
$i=1,\dots,n$, set
$$
R_i:=T_i\setminus Q_i,
\qquad
\Sigma_i:=\partial Q_i\cap\operatorname{int}T_i\,.
$$ 
Let $w_i$ and $z_i$ denote the mixed torsion functions of $T_i$
and $Q_i$, respectively, with Dirichlet condition on
$[A_i,A_{i+1}]$. Then, in the limit  
as $a\to a^*$, it holds
\begin{equation}\label{f:strategy-asymptotics}
\int_{R_i}w_i\,dx
=\bigl(w^*(0)+o(1)\bigr)|R_i|,
\qquad
\inf_{\Sigma_i}z_i=w^*(0)+o(1),
\end{equation}
where the second estimate is understood whenever
$\Sigma_i\ne\varnothing$, and $w^*$ is the mixed torsion
function of a fundamental triangle of $D^*$, with
Dirichlet condition on the side lying in $\partial D^*$.
\end{theorem}

Theorems \ref{t:truncation} and  \ref{t:strategy-asymptotics}
are proved in Section~\ref{s:proof-strategy-lemmas}. 
 
 \medskip 
\subsection{Optimality of isosceles triangles for the mixed torsional rigidity}\label{s:isosceles}
 
Let $T \subset \R^2$ be a triangle. Let us fix a Dirichlet side  for $T$, that we denote by $\Gamma$. The vertex of $T$ opposite to $\Gamma$ will be called the (NN)-vertex. 

We consider the 
problem of maximizing the mixed torsional rigidity of $T$ (with Dirichlet side $\Gamma$) under an area constraint within the class 
 $\mathcal A_\alpha$ of triangles with a fixed interior angle $\alpha \in (0, \frac{\pi}{2}]$ 
 at the (NN)-vertex.  
  In terms of 
the scale--invariant energy 
\begin{equation}\label{f:energy}
\mathcal E_{mix}(T):=\frac{\T_{\rm mix}(T)^{\frac14}}{|T|^{\frac12}}\,,
\end{equation}
the problem becomes: 
\begin{equation}\label{f:sup}
\sup\big\{\mathcal E_{mix}(T):T\in\mathcal A_\alpha\big\}\,.
\end{equation}

We prove (see Section \ref{sec:isosceles}):

\begin{theorem}\label{t:tmax}
For every $\alpha\in(0,\frac{\pi}{2}]$, the supremum in \eqref{f:sup} is attained uniquely at isosceles triangles.
\end{theorem}

\medskip

\subsection{A concavity property of the mixed torsional rigidity of triangles}\label{s:local-concavity}
Let $\mathcal T (\alpha)$ denote the mixed torsional rigidity  of 
the isosceles triangle with equal sides $1$ and vertex angle $\alpha$, defined according to \eqref{f:tmixvar}, by choosing as  Dirichlet side the base opposite to the vertex  angle $\alpha$.
Set 
\begin{equation}\label{f:effe} 
F(\alpha):=\frac{\alpha^2}{\sin^2\alpha}\,\mathcal T (\alpha) \qquad    \forall \alpha \in (0 , \pi)   \,.
\end{equation} 
We prove (see Section \ref{sec:concavity}):

\begin{theorem}\label{t:concavity}
The function  $F$    satisfies
\begin{equation}\label{f:conc-curvature}
F''(\alpha)<0    \qquad \forall \alpha \in (0 , \pi)\,. 
\end{equation}

\end{theorem}

\medskip
\subsection{Conclusion}\label{s:combining} 

Let $D$ be a convex polygon with $|D| = |D ^*|$ and vertex tuple $a=(A_1,\ldots,A_{n})$ sufficiently close
to $a^*=(A_1^*,\ldots,A_{n}^*)$.
Let $\{T_i\}_{i=1}^{n}$ and $\{Q_i\}_{i=1}^{n}$  be respectively the triangular covering and 
the convex polygonal partition of $D$ given by Theorem~\ref{t:cover}. 
Let $\T_{mix} (T_i)$ be the mixed torsional rigidity of $T_i$ defined according to \eqref{f:tmixvar}, with Dirichlet condition on the side  $[A_i, A_{i+1}]$.
 Set  
$$
S:=\sum_{i=1 } ^ { n} |T_i| 
\,, 
\qquad \Delta:=S-|D| .
$$
We shall prove that, for $a$ sufficiently close to $a^*$, we have 
\begin{equation}\label{f:strategy-jensen}
\T(D)\le \Big (\frac{|D|}{S}\Big )^2 \sum_{i=1}^{n}  \T_{\rm mix}(T_i)
\end{equation}

To this end, we use the polygonal convex partition $\{Q_i\} _{i=1} ^ { n}$ given by Theorem~\ref{t:cover}. Denote by 
$\T_{\rm mix}(Q_i)$  the mixed torsional rigidity of 
$Q_i$, with Dirichlet condition on the base $[A_i, A_{i+1}]$. 
Releasing the internal interfaces gives the estimate
\begin{equation}\label{f:strategy-partition}
\T(D)\le \sum_{i=1}^n \T_{\rm mix}(Q_i).
\end{equation}

Since $D$ is sufficiently close to $D^*$, each triangle $T_i$
is acute. Moreover,  $Q_i$ is convex and contains the whole
base $[A_i,A_{i+1}]$. 
Let $w_i$ and $z_i$ denote the mixed torsion functions of $T_i$ and $Q_i$, respectively, 
and set
$
R_i:=T_i\setminus Q_i$ and 
$\Sigma_i:=\partial Q_i\cap\operatorname{int}T_i$,   if $Q_i \neq T_i$,  Theorem~\ref{t:truncation} gives
\begin{equation}\label{f:strategy-cellloss}
\T_{\rm mix}(Q_i)
\le \T_{\rm mix}(T_i)
-\int_{R_i}w_i\,dx-|R_i|\inf_{\Sigma_i}z_i.
\end{equation}
If $Q_i=T_i$, the mixed torsional rigidities coincide
and there is no loss term.

In the limit as $a\to a^*$, the sets $R_i$
and the interfaces $\Sigma_i$ shrink to the origin: 
Theorem~\ref{t:strategy-asymptotics} yields
\begin{equation}\label{f:strategy-stronger}
\T_{\rm mix}(Q_i)
\le \T_{\rm mix}(T_i)
-\bigl(2w^*(0)+o(1)\bigr)|R_i|.
\end{equation}
This inequality also holds when $Q_i=T_i$,
since then $|R_i|=0$.

\medskip 
Summing over $i$, by using \eqref{f:strategy-partition} and  the fact that 
$\sum_{i=1}^n|R_i|=S-|D|=\Delta$, 
 we obtain  \begin{equation}\label{f:c}
\T(D)\le \sum_{i=1}^{n}  \T_{\rm mix}(T_i)  -\bigl(2w^*(0)+o(1)\bigr)\Delta\,. 
\end{equation}
We now observe that \eqref{f:c} implies \eqref{f:strategy-jensen}, for $a$ close to $a^*$. Indeed, 
let us write for brevity 
$$H :=  \sum_{i=1}^{n}  \T_{\rm mix}(T_i)  \,, \qquad H^* := n \T_{\rm mix}(T^*) \,.$$
In the limit as $a \to a ^*$,   the right--hand side of \eqref{f:strategy-jensen}  has asymptotics
\[
\left(\frac{|D|}{S}\right)^2H
=H-\left(2\frac{H}{S}-\frac{H}{S^2}\Delta\right)\Delta
=H-\left(2\frac{H^*}{|D|}+o(1)\right)\Delta\,.
\]
Thus, whenever $\Delta>0$, it is strictly larger than
the right-hand side of \eqref{f:c},   since, by
\cite[Corollary~2(i)]{BF26}, $w^*$ attains its unique global
maximum at the (NN)-vertex $0$, and therefore  
$$
\frac{H^*}{|D|}
=\frac{1}{|T^*|}\int_{T^*}w^*
<w^*(0).
$$
If $\Delta=0$, then $S=|D|$ and
\eqref{f:strategy-jensen}, for $a$ sufficiently close to $a^*$,  follows directly from \eqref{f:c}. 

To conclude the proof, let us show that the result follows from \eqref{f:strategy-jensen}.  

By  Theorem~\ref{t:tmax},  $$
\T_{\rm mix}(T_i)
\le \frac{4|T_i|^2}{\sin^2\alpha_i} \mathcal T (\alpha_i) 
=\frac{S^2}{\pi^2}F(\alpha_i),
$$
where the last equality follows from the proportionality of the system $\{ T_i \} _ { i= 1} ^n$ and the definition of 
$F$ in \eqref{f:effe}.

Summing the preceding estimate over $i$ and using \eqref{f:strategy-jensen}, we  obtain  
$$
\T(D)\le \Big ( \frac{|D|}{\pi} \Big )  ^2 \sum_{i=1}^{n} F(\alpha_i).
$$
Since $D$ is sufficiently close to the regular polygon, by Theorem \ref{t:concavity} there exist
$\eta_n>0$ and $\kappa_n>0$ such that all the angles $\alpha_i$ belong
to
$\left[\frac{2\pi}{n}-\eta_n,\frac{2\pi}{n}+\eta_n\right]$
and
$$
F(\alpha_i)\leq
F\big (\frac{2\pi}{n}\big )
+F'\big(\frac{2\pi}{n}\big)
\big(\alpha_i-\frac{2\pi}{n}\big)
-\frac{\kappa_n}{2}
\big(\alpha_i-\frac{2\pi}{n}\big)^2.
$$
Summing over $i=1,\ldots,n$ and using
$\sum_{i=1}^n\big(\alpha_i-\frac{2\pi}{n}\big)=0$, we obtain \begin{equation}\label{f:conc-jensen-main}
\sum_{i=1}^{n}F(\alpha_i)
\le n F \big  (\frac {2\pi} n \big  )-\frac{\kappa_n}{2}
\sum_{i=1}^{n} \big  (\alpha_i- \frac {2\pi} n \big  )^2 \,. 
\end{equation}

 Hence, 
\begin{equation}\label{f:pnonreg} 
\T(D)\le \frac{|D|^2}{\pi^2} n F \big (   \frac {2\pi} n  \big  )\,.
 \end{equation}  
By \eqref{f:conc-jensen-main},  equality holds if and only if  $\alpha_i=   \frac {2\pi} n   $ for all $i$.   
For the regular polygon $D^*$ having the same area as $D$, the $n$ fundamental triangles 
 have area $|D|/n$, and vertex angle $ \frac {2\pi} n$.  Therefore, 
\begin{equation}\label{f:preg}
\T(D^*)=\Big ( \frac{|D|}{\pi} \Big )  ^2 n F\big (   \frac {2\pi} n  \big  ).
\end{equation}

 We infer that 
$\T(D)\le \T(D^*)$, 
with strict inequality unless  $\alpha_i=    \frac {2\pi} n   $ for all $i$.   Moreover, equality in the triangular optimization
is attained only when every $T_i$ is isosceles. 
 Since $\alpha_i=   \frac {2\pi} n    $ for every $i$ and since the covering is
proportional, all the triangles $T_i$ have the same area.
Being also isosceles with the same vertex angle, they are
therefore congruent. Hence   $D$ is regular. 
\qed   

\section{Proof of Theorem~\ref{t:cover}}\label{s:proof-cover}

For  clarity, we divide the proof into several steps. 
\medskip 

 \underbar{Step 1.} Let us show that, for a family of triangles
$\{T_i\}_{i=1}^n$ with bases $[A_i,A_{i+1}]$,
and sufficiently close to the fundamental system of $D^*$,
the condition that $\{T_i\}_{i=1}^n$ form an admissible proportional system
is equivalent to requiring that their base angles and the proportionality
ratio $R$ satisfy a suitable system of equations.  More precisely, let us denote by 
$V_i \in D$ the apex of $T_i$, so that
$$T_i= \triangle (A_i, A _{i+1}, V_i ) \qquad \forall i = 1, \dots, n\,, $$
and  let us introduce the vector $(x, y,  R) \in \R ^ { 2n} \times \R$  with components 
$$x_i:= \angle{ (A_{i+1}, A_i, V_i) } \,, \qquad  y _i := \angle ( A_i, A_{i+1} , V_i )\,, \qquad 
R:= \frac{\alpha_i}{| T_i|} \,. $$
Moreover, for every $i = 1, \dots, n$, we denote by   $ \widehat A_{i} $ the interior angle of $D$ at $A_i$, and by  
 $ s_i:=  \mathcal H ^ 1 ( [A_i, A_{i+1}] )$. 
 
 Then $\{T_i\}_{i=1}^n$ is an admissible proportional system for $D$
if and only if $(x,y,R)$ solves the following system: 

\begin{align}
& x_{i+1}+y_i
    = \widehat{A}_{i+1},
    \qquad i=1,\ldots,n,
    \label{bbfp02}\\
    \noalign{\medskip} 
& \frac{2(\pi-x_i-y_i)\sin(x_i+y_i)}
     {s_i^2\sin x_i\sin y_i}
    = R,
    \qquad i=1,\ldots,n.
    \label{bbfp01}
\end{align}  

Indeed, condition \eqref{bbfp02} is equivalent, for systems sufficiently
close to the fundamental one, to requiring that $V_i$ and $V_{i+1}$ lie
on the same inward ray $\ell_{i+1}$ emanating from $A_{i+1}$.
Consequently, $T_i$ and $T_{i+1}$ share a boundary segment which is a
full side of at least one of them, and in particular we have
$\sum_{i=1}^n \alpha_i = 2\pi$.

On the other hand, since $\alpha _ i= \pi - x_i - y _i$, and     $|T_i| = \frac {s_i^2\sin x_i\sin y_i} 
 {2\sin(x_i+y_i)}$, equation  \eqref{bbfp01} corresponds to the proportionality condition $\frac{\alpha_i}{|T_i|} =R$. 
 
 \medskip
 \underbar{Step 2}. We claim that, if   $\{T_i \} _{i = 1 } ^n$ is an admissible proportional system of $D$, 
it yields a covering of $D$  provided the angular variables introduced above
  satisfy the equation
  \begin{equation}\label{bbfp03}  
    \prod_{i=1}^{n}\sin x_i= \prod_{i=1}^{n}\sin y_i \,.
 \end{equation} 
 
Indeed, thanks to property (b) recalled in Section \ref{s:local-covering},  to prove the claim it is enough to show that
 \begin{equation}\label{f:pNN} 
  p _n = \frac{\prod_{i=1}^{n}\sin y_i } {\prod_{i=1}^{n}\sin x_i }   \,. 
  \end{equation}  
  The above equality follows immediately from the recursive definition \eqref{f:recursive} of  the weights
 $\{p_i \} _{i = 1 } ^n$, and of the elementary geometric equality  
 $$ \frac{d(z,q_{i-1})}{d(z,q_i)}  = \frac{\sin y_{i-1} }{\sin x_i}  \qquad \forall z \in \ell _i \setminus  \{A_i\} \,.$$ 
 
 \medskip
 \underbar{Step 3}. In view of Steps 1 and 2,  in order to prove the existence and uniqueness parts in the statement
 of Theorem~\ref{t:cover}, 
 it is enough to show that, for $a$ sufficiently close to $a^*$,  the system of $2n+1$ equations
 \eqref{bbfp02}--\eqref{bbfp01}--\eqref{bbfp03}   admits a unique solution.  
 
 We point out that, in this system, the parameters 
   $ \widehat A_{i} $ and
 $ s_i$ are uniquely determined by the given vector $a = (A_1, \dots, A _n) \in \R ^ { 2n}$, 
 while  the unknown is the vector   $(x, y, R ) \in \R ^ {2n +1} $.

We consider the function 
  $$ \Phi  : \R^{2n+1} \times \R^{2n} \to \R^{2n+1} \,,$$   
depending on the variables $(x_i,y_i,R)\in \R^{2n+1}$ and $a\in \R^{2n}$,  whose components 
 $(\Phi ^ 1, \Phi ^ 2, \Phi ^ 3) \in \R ^n \times \R ^n \times \R$ are 
defined by 
$$
\begin{aligned}
& \Phi ^ 1 (x_i,y_i,R, a)   := 2(\pi - x_i-y_i)\sin (x_i+y_i)-Rs_i^2\sin x_i \sin y_i 
\\ 
& \Phi ^ 2 (x_i,y_i,R, a)   := x_{i+1}+y_i - \widehat A_{i+1} 
\\
& \Phi ^ 3 (x_i,y_i,R, a)   := \prod_{i=1}^{n}\sin x_i- \prod_{i=1}^{n}\sin y_i
\end{aligned}
$$ 
When $a = a ^*$,  the vectors $(x, y)$ and $R$ take respectively the values   $(x^* , y ^*)$ and $R^*$ given (componentwise)  by 
$$
x_i^*=y_i^*=\frac{n-2}{2n}\pi\,, \qquad R^*= \frac{4\pi}{n\sin \frac{2\pi}n}\,,$$ 
 and 
 $$ \Phi   (x ^*, y ^*, R ^ *, a ^* ) = 0\,.$$
 Thus, the existence and uniqueness of a solution 
 for $a$ close to $a ^*$ follows from the implicit function theorem, provided the determinant of 
 the Jacobian matrix of $  \Phi  $  in the first $2n+1$ variables, computed at $(x ^*, y^*, R^*, a^*)$,   is nonzero.
This Jacobian matrix is given by 
 $$D_{(x^*,y^*,R^*)}  \Phi   = 
\begin{pmatrix}
c_n I_n &c_n I_n &b_n 1_n^T\\
J_n & I_n &0_n ^T\\
{d_n}   1_n& -  {d_n}  1_n & 0
\end{pmatrix},
$$
 where:  $I_n$ and $J_n$  are $(n\times n)$ matrices given respectively by the Identity matris and 
 \[
J_n :=
\begin{pmatrix}
0 & 1 & 0 & \cdots & 0 \\
0 & 0 & 1 & \cdots & 0 \\
\vdots & & \ddots & \ddots & \vdots \\
0 & \cdots & 0 & 0 & 1 \\
1 & 0 & \cdots & 0 & 0
\end{pmatrix},
\]
$1_n$ and $0 _n$ are the vectors of $\R ^n$ with components 
$$1_n := [1, 1, ... 1] \,, \qquad  0_n:= [ 0, 0, \dots 0] \,,$$  
 and finally $b _n$,  $c_n$,  and $d_n$   are the scalars 
$$
\begin{aligned}
 & b_n : = - (s^*) ^2 \sin ^2 x^* <0
 \\ 
  & c_n:= -2 \sin 2x^*+2(\pi-2x^*)\cos 2x^*-R^*(s^*) ^2 \cos x^* \sin x^*<0\,.
  \\ 
  & 
  d_n := \cos x^*\,(\sin x^*)^{n-1} > 0.   
\end{aligned} 
  $$  
Using the specific structure of $J_n$, some elementary computations give:   
$$\det D_{(x^*,y^*,R^*)} \Phi    = -2n^2 b_n   c_n^{n-1}   d_n    \not=0.$$

As a  direct consequence of our proof using the implicit function theorem, we obtain the continuous dependence of the system $\{ T _i \}_{i = 1} ^n$ on the vector $a$, and hence the convergence  of the apices $V_i$ to $0$ as $a \to a ^*$.

\medskip
\underbar{Step 4.} We finally justify the part of the statement concerning the cells $Q_i$.  
For every $j$, set $\varphi_j(x):=p_jd(x,q_j)$. Since $D$ lies in one
of the closed half-planes determined by the supporting line $q_j$, the
function $d(\cdot,q_j)$ is affine on $D$. Hence each $\varphi_j$ is
affine on $D$, and the convexity of $Q_i$ follows immediately from its definition, since it  is an intersection of half-planes. 
 The fact that $Q_i$ contains the segment $[A_i , A_{i+1}]$ is immediate, since on such segment $d(x, q_i)$ vanishes. 
 To prove the inclusion $Q_i\subset T_i$, notice that, by definition of the weights, we have 
$$\varphi_{i-1}=\varphi_i \quad\text{on }\ell_i,
\qquad
\varphi_i=\varphi_{i+1} \quad\text{on }\ell_{i+1}.
$$
Moreover, at every point in the relative interior of
$[A_i,A_{i+1}]$ one has
$\varphi_i=0<\varphi_{i-1},\varphi_{i+1}$.

Now, $\varphi_i-\varphi_{i-1}$ is affine on $D$ and vanishes on $\ell_i$.
Since it is negative in the half-plane bounded by the line containing
$\ell_i$ to which $\operatorname{int} T_i$ belongs, it is positive in the
other half-plane. Hence $\varphi_{i-1}<\varphi_i$ in the open half-plane
bounded by the line containing $\ell_i$ and disjoint from
$\operatorname{int} T_i$.
Similarly, $\varphi_i-\varphi_{i+1}$ is affine and vanishes on
$\ell_{i+1}$, so that $\varphi_{i+1}<\varphi_i$ in the open half-plane
bounded by the line containing $\ell_{i+1}$ and disjoint from
$\operatorname{int} T_i$. 
 
Since every point of $D\setminus T_i$ belongs to at least one of
these two half-planes, it cannot belong to $Q_i$. Hence
$Q_i\subset T_i$.  

 Finally, it follows directly from their definition that the sets $Q_i$
cover $D$ and have pairwise disjoint interiors, so that they form a
partition of $D$ up to their common boundaries. \qed

\bigskip 
\section{Proofs of Theorems \ref{t:truncation} and
\ref{t:strategy-asymptotics}}
\label{s:proof-strategy-lemmas}

\begin{proof}[Proof of Theorem \ref{t:truncation}]
Let $\Gamma=[A,B]$ be the Dirichlet side of $T$, let $V$ be the
opposite vertex, and denote by $S'$ and $S''$ the two Neumann sides
of $T$, with outward unit normals $\eta'$ and $\eta''$, respectively.
We first prove that
\begin{equation}\label{f:truncation-flux-sign}
\partial_\nu w\geq0
\qquad\text{on }\Sigma,
\end{equation}
where $\nu$ denotes the outward unit normal to $Q$.

 Since $T$ is an acute triangle, it satisfies the assumptions of
\cite[Corollary~2(iii)]{BF26}. Therefore, the mixed torsion function
$w$ is strictly increasing along every segment contained in $T$ and
orthogonal to a Neumann side, when oriented towards that side.   Hence
$$
\nabla w\cdot\eta'\geq0,
\qquad
\nabla w\cdot\eta''\geq0
\qquad\text{in }T.
$$

Let $L$ be one of the line segments forming $\Sigma$, and let $\nu$
be the outward unit normal to $Q$ along $L$. Since $Q$ is convex and
contains the whole base $\Gamma$, the segment $\Gamma$ lies in the
closed half-plane bounded by the line containing $L$ and containing
$Q$. The vertex $V$ lies strictly in the opposite half-plane.
Indeed, otherwise, by convexity of that half-plane,
$T=\triangle(A,B,V)$ would be entirely contained in it, contradicting
$L\subset\operatorname{int}T$.

Since $\nu$ points towards the half-plane containing $V$, we have
$$
\nu\cdot(A-V)<0,
\qquad
\nu\cdot(B-V)<0.
$$
These two inequalities imply that $\nu$ belongs to the positive cone
generated by $\eta'$ and $\eta''$. Hence, there exist
$\lambda',\lambda''\geq0$ such that
$$
\nu=\lambda'\eta'+\lambda''\eta''.
$$
Consequently,
$$
\partial_\nu w
=
\lambda'\,\nabla w\cdot\eta'
+
\lambda''\,\nabla w\cdot\eta''
\geq0
\qquad\text{on }L,
$$
and \eqref{f:truncation-flux-sign} follows.

We now derive an exact identity for the loss of mixed torsional
rigidity. Testing the weak equation for $z$ with $w|_Q$ gives
$$
\int_Q\nabla z\cdot\nabla w=\int_Qw,
$$
whereas, by integration by parts for $w$ in $Q$,
$$
\int_Q\nabla w\cdot\nabla z
=
\int_Qz+\int_\Sigma z\,\partial_\nu w.
$$
Hence
$$
\int_Qw-\int_Qz
=
\int_\Sigma z\,\partial_\nu w,
$$
and therefore
\begin{equation}\label{f:truncation-identity}
\T_{\rm mix}(T)-\T_{\rm mix}(Q)
=
\int_R w\,dx
+
\int_\Sigma z\,\partial_\nu w.
\end{equation}

Moreover, integrating $-\Delta w=1$ over $R$ and recalling that
$\partial_\nu w=0$ on the Neumann sides of $T$, we get
\begin{equation}\label{f:truncation-flux-mass}
\int_\Sigma\partial_\nu w=|R|,
\end{equation}
 because the outward unit normal to $R$ along $\Sigma$ is $-\nu$. 

By \eqref{f:truncation-flux-sign},
\eqref{f:truncation-identity}, and
\eqref{f:truncation-flux-mass}, we obtain \eqref{r:truncation-bounds}.  
\end{proof}

\bigskip

\begin{proof}[Proof of Theorem \ref{t:strategy-asymptotics}]
Fix $i\in\{1,\ldots,n\}$ and let $a_h\to a^*$. 

For simplicity, write
$T_h,Q_h,R_h,\Sigma_h,w_h,z_h$ for the corresponding objects.
Then $T_h$ and $Q_h$ converge to the fundamental triangle $T^*$,
whereas $R_h$ and $\Sigma_h$ shrink to the origin, which is the vertex corresponding
to the center of the regular polygon.

Let $F_h:T^*\to T_h$ be the affine map sending the vertices of $T^*$
onto those of $T_h$. Since $F_h\to{\rm Id}$, the pulled-back
solutions $w_h\circ F_h$ converge to $w^*$ in $H^1(T^*)$ and satisfy
a uniform $H^2$ estimate on $T^*$. Hence, by compactness, the
convergence is uniform on $\overline{T^*}$. Since $R_h$ shrinks to
$0$ and $w^*$ is continuous at $0$, it follows that
$$
\sup_{R_h}|w_h-w^*(0)|=o(1).
$$
Therefore
\begin{equation}\label{f:asymptotic-cap}
\int_{R_h}w_h
=
\bigl(w^*(0)+o(1)\bigr)|R_h|.
\end{equation}

It remains to prove the assertion concerning $\Sigma_h$.
 Recall that this assertion is needed only for those $h$ such that
$\Sigma_h\neq\varnothing$. For such $h$, we prove that
$$
\sup_{x\in\Sigma_h}|z_h(x)-w^*(0)|=o(1),
$$
which immediately yields
$$
\inf_{\Sigma_h}z_h=w^*(0)+o(1).
$$

We now establish a uniform modulus of continuity for $z_h$.

Let $\Gamma_h$ denote the Dirichlet side of $Q_h$,  and 
let $\mathcal R_h$ be the reflection with respect to the line containing
$\Gamma_h$. 
Since $Q_h\subset T_h$ and contains the whole base $\Gamma_h$, the two angles of $Q_h$ at the endpoints of $\Gamma_h$ are bounded above by the corresponding base angles of $T_h$. Since $T_h$ converges to the acute triangle $T^*$, these angles are strictly smaller than $\pi/2$ for $h$ large. Therefore, the symmetrized polygon 
$\widehat Q_h:=Q_h\cup\mathcal R_h(Q_h)$
is convex. 

We denote by $\widehat z_h \in H^1(\widehat Q_h)$ the odd extension of $z_h$ across $\Gamma_h$, which solves the Neumann problem 
$$
-\Delta\widehat z_h=f_h
\quad\text{in }\widehat Q_h,
\qquad
\partial_\nu\widehat z_h=0
\quad\text{on }\partial\widehat Q_h,
$$
where
$$
f_h
=
\begin{cases}
1, & \text{in }Q_h,\\
-1, & \text{in }\mathcal R_h(Q_h).
\end{cases}
$$
In particular,
$
\int_{\widehat Q_h}f_h=0$, so that the compatibility condition for the Neumann problem is satisfied. 

By the Lipschitz estimate for the Poisson--Neumann problem on convex
domains, see \cite[Section 2]{mazya}, for every fixed $q>2$ we have
$$
\|\nabla\widehat z_h\|_{L^\infty(\widehat Q_h)}
\leq
c(q)\,C_{\widehat Q_h}^{-1}
|\widehat Q_h|^{(q-2)/(2q)}
\|f_h\|_{L^q(\widehat Q_h)},
$$
where $C_{\widehat Q_h}$ denotes the constant in the relative
isoperimetric inequality for $\widehat Q_h$.

In our case, since $\|f_h\|_{L^q(\widehat Q_h)}
=|\widehat Q_h|^{1/q}$, the right-hand side is
$$
c(q)\,C_{\widehat Q_h}^{-1}|\widehat Q_h|^{1/2}.
$$

Since the cells $Q_h$ are defined by a fixed finite family of affine
inequalities whose coefficients depend continuously on $a_h$, and
$a_h\to a^*$, the convex polygons $\widehat Q_h$ satisfy, for $h$
large, a uniform cone condition with fixed parameters. Moreover,
their diameters are uniformly bounded. Hence, by \cite[Corollary~3.2]{thomas}, the relative isoperimetric inequality holds on the family
$\{\widehat Q_h\}$ with a constant independent of $h$. It follows that, for some $C>0$
independent of $h$,
\begin{equation}\label{f:uniform-lipschitz-zh}
\|\nabla\widehat z_h\|_{L^\infty(\widehat Q_h)}
\leq C.
\end{equation}

We next identify the limit of $z_h$ away from $0$. Let $K$ be a
compact subset of $\overline{T^*}\setminus\{0\}$. For $h$ large,
$K\cap T^*$ is contained in $Q_h$, up to the harmless displacement
of the boundary produced by $F_h$. The  preceding estimates therefore
yield compactness on every such $K$.

Let $\varphi$ be smooth on $\overline{T^*}$, vanish on the Dirichlet
side, and vanish in a neighbourhood of $0$. Extend $\varphi$
smoothly to a neighbourhood of $\overline{T^*}$ and set
$$
\varphi_h:=\varphi\circ F_h^{-1}.
$$
For $h$ large, $\varphi_h$ is an admissible test function for the
weak problem defining $z_h$, so that
$$
\int_{Q_h}\nabla z_h\cdot\nabla\varphi_h
=
\int_{Q_h}\varphi_h.
$$
Passing to a subsequence and then to the limit, using
$F_h\to{\rm Id}$ and $|Q_h\triangle T^*|\to0$, we obtain
$$
\int_{T^*}\nabla z\cdot\nabla\varphi
=
\int_{T^*}\varphi.
$$
The uniform Lipschitz estimate and the convergence of the Dirichlet
sides also imply that $z$ has zero trace on the Dirichlet side of
$T^*$. 

Since a point has zero $H^1$-capacity in dimension two, test
functions vanishing in a neighbourhood of $0$ are dense in the
natural energy space of the mixed problem. Hence $z$ is the mixed
torsion function $w^*$. By uniqueness of the limit,
$
z_h\to w^*
$
locally uniformly on
$\overline{T^*}\setminus\{0\}$.

It remains only to transfer the convergence to the shrinking interfaces.
Recall that this assertion is required only for those $h$ such that
$\Sigma_h\neq\varnothing$.

Fix $\varepsilon>0$. By continuity of $w^*$ and
\eqref{f:uniform-lipschitz-zh}, choose
$y\in\operatorname{int}T^*$ sufficiently close to $0$ so that
$$
C|y|+|w^*(y)-w^*(0)|<\frac{\varepsilon}{2}.
$$
For $h$ large, $y\in Q_h$, and the local uniform convergence gives
$$
|z_h(y)-w^*(y)|<\frac{\varepsilon}{4}.
$$
Moreover, since $\Sigma_h$ shrinks to $0$, uniformly for
$x\in\Sigma_h$ we have
$$
|x-y|\leq |y|+o(1).
$$
Therefore, by \eqref{f:uniform-lipschitz-zh}, for every
$x\in\Sigma_h$,
$$
\begin{aligned}
|z_h(x)-w^*(0)|
&\leq
|z_h(x)-z_h(y)|
+
|z_h(y)-w^*(y)|
+
|w^*(y)-w^*(0)|\\
&\leq
C|x-y|
+
|z_h(y)-w^*(y)|
+
|w^*(y)-w^*(0)|
<\varepsilon
\end{aligned}
$$
for $h$ sufficiently large. Hence
$$
\sup_{x\in\Sigma_h}|z_h(x)-w^*(0)|=o(1),
$$
and consequently
$\inf_{\Sigma_h}z_h=w^*(0)+o(1)$.
Together with \eqref{f:asymptotic-cap}, this proves
\eqref{f:strategy-asymptotics}.
\end{proof}

\section{Proof of Theorem \ref{t:tmax}} \label{sec:isosceles}

Theorem \ref{t:tmax}  follows directly from following two facts:  
 problem \eqref{f:sup} admits a solution (see 
Proposition \ref{p:existence}),  and the isosceles triangles are the unique triangles 
in $\mathcal A _\alpha$  which are critical  for the energy $\mathcal E _{mix}$ 
under suitable shape perturbations (see Proposition \ref{p:reflection}).  
 \begin{proposition}\label{p:existence}
The supremum in \eqref{f:sup} is finite and is attained.
 \end{proposition}
\begin{proof}   
Let $\{T_h\}$ be a maximizing sequence for problem \eqref{f:sup}, and
assume with no loss of generality that $|T_h|=1$ for every $h$.
Up to rigid motions, we may assume that $T_h = \triangle(O_h,A_h,B_h ) $, where
$$
 O_h=0,\quad A_h=a_h e_1,\quad B_h=b_h e_\alpha,
$$
where
$e_1=(1,0)$ and   $e_\alpha=(\cos\alpha,\sin\alpha)$, 
and the Dirichlet side $\Gamma_h$ is the segment $[A_h,B_h]$.
Exchanging $A_h$ and $B_h$ if necessary, we may assume that
$a_h\geq b_h>0$. Since $|T_h|=1$, we have
\begin{equation}\label{f:area-normalization}
 a_hb_h\sin\alpha=2.
\end{equation}

In what follows, we fix $h$ and, for brevity, write $a, b, T$ in place of 
$a_h, b _h, T_h$. 
The map
$(s,t)\longmapsto se_1+te_\alpha$
parametrizes $T$ over the set
$$
 0<s<a,\qquad
 0<t<\ell(s):=b\left(1-\frac{s}{a}\right),
$$
and has constant Jacobian $\sin\alpha$. Notice that, for every
$s\in(0,a)$, the point
$
 se_1+\ell(s)e_\alpha
$
belongs to the Dirichlet side $\Gamma_h$.

For every $u\in H^1_{\Gamma_h}(T)$,  letting 
$U(s,t):=u(se_1+te_\alpha)$, 
we have $U(s,\ell(s))=0$, and hence
$$
 U(s,t)
 =-\int_t^{\ell(s)}\partial_tU(s,r)\,dr \qquad \text{ for a.e. } (s, t)\,.
$$
 By the Cauchy--Schwarz inequality,
$$
 |U(s,t)|^2
 \leq
 \bigl(\ell(s)-t\bigr)
 \int_t^{\ell(s)}|\partial_tU(s,r)|^2\,dr.
$$

Integrating with respect to $t$ and using Fubini's theorem, we obtain
$$
\begin{aligned}
 \int_0^{\ell(s)}|U(s,t)|^2\,dt
 &\leq
 \int_0^{\ell(s)}
 |\partial_tU(s,r)|^2
 \left[\int_0^r\bigl(\ell(s)-t\bigr)\,dt\right]dr
 \\
 &\leq
 \frac{\ell(s)^2}{2}
 \int_0^{\ell(s)}|\partial_tU(s,r)|^2\,dr
 \\
 &\leq
 \frac{b^2}{2}
 \int_0^{\ell(s)}|\partial_tU(s,r)|^2\,dr.
 \end{aligned} 
$$
Since
$$
 \partial_t U(s,t)
 =\nabla u(se_1+te_\alpha)\cdot e_\alpha,
$$
integration with respect to $s$, taking into account the constant
Jacobian $\sin\alpha$, gives
$$
 \int_Tu^2
 \leq
 \frac{b^2}{2}\int_T|\nabla u\cdot e_\alpha|^2
 \leq
 \frac{b^2}{2}\int_T|\nabla u|^2.
$$
Since $|T|=1$, the Cauchy--Schwarz inequality yields
$$
 \left|\int_Tu\right|^2
 \leq \int_Tu^2
 \leq
 \frac{b^2}{2}\int_T|\nabla u|^2.
$$
It follows from the variational characterization of the mixed
torsional rigidity that
\begin{equation}\label{f:torsion-degeneration-estimate}
 \T_{\rm mix}(T_h)\leq\frac{b_h^2}{2}.
\end{equation}
Moreover, since $b_h\leq a_h$, by \eqref{f:area-normalization} we have
$$
 b_h^2\leq a_hb_h=\frac{2}{\sin\alpha},
$$
and therefore
$$
 \T_{\rm mix}(T_h)\leq\frac{1}{\sin\alpha}.
$$
This proves that the supremum in \eqref{f:sup} is finite.

We now exclude degeneration of the maximizing sequence. Assume by
contradiction that the minimal width of $T_h$ is infinitesimal as
$h\to+\infty$. In view of \eqref{f:area-normalization} and of the
inequality $a_h\geq b_h$, this is equivalent, up to a subsequence, to
$$
 a_h\longrightarrow+\infty,
 \qquad
 b_h=\frac{2}{a_h\sin\alpha}\longrightarrow0.
$$
Estimate \eqref{f:torsion-degeneration-estimate} then gives
$\T_{\rm mix}(T_h)\to 0$. 
Since $|T_h|=1$, we consequently have
$\mathcal E_{mix}(T_h)\to 0$, contradicting the maximizing property of the sequence. 
Thus, the minimal width of $T_h$ is bounded away from zero, and hence, up to a subsequence,
$T_h$ converges in Hausdorff distance to a
nondegenerate triangle
$ T$
which belongs to $\mathcal A_\alpha$ and has area one. By the
continuity of $\mathcal E_{mix}$ under Hausdorff convergence,
$T$ is a solution to problem \eqref{f:sup}.
\end{proof}

\bigskip

We now identify isosceles triangles as the unique stationary triangles for the energy $\mathcal E_{mix}$ under a specific shape perturbation, 
which consists of rotating the Dirichlet side about its midpoint. 
To be more precise, given $T \in \mathcal A _\alpha$, 
let us denote by  $\Gamma  = [A, B]$ its Dirichlet side  and by  $O$ its (NN)-vertex. 
Let us denote by $T _{\e}$ the one-parameter family of triangles $\triangle(O, A_{\e} ,B _{\e})$ 
constructed as follows: 
 \begin{itemize}
\item[--] the line through $A_{\e},B_{\e}$ is obtained by rotating the line through $A,B$ by an angle $\e$  about the midpoint 
$M$ of $[A, B]$;
\item[--]  $A_{\e}$ and $B_{\e}$ lie respectively on the straight lines through $O,A$  and through $O,B$;
\item[--] the direction of the rotation is determined by the convention that, for $\e>0$, $B_{\e}$ lies on the segment $[O, B]$, while for $\e<0$, $A_{\e}$ lies on the segment $[O, A]$.
\end{itemize}
 
We say that $T$ is {\it critical for $\mathcal E _{mix}$ under rotation of $\Gamma$} if  
\begin{equation}\label{f:statmix} 
\frac{d}{d\e }\Big|_{\e=0}\mathcal E_{mix}(T_\e)   = 0\,.
\end{equation} 
Notice that the above equality  can also be rephrased in terms of the mixed torsion function $w _T$. 
 Indeed, the argument in the proof of \cite[Lemma 26]{FV19}
applies to the present mixed problem as well. The reason is that,  under the
above perturbation, the two Neumann sides remain on the same
supporting lines, so that the normal component of the deformation
field vanishes there, and only the moving Dirichlet side contributes
to the shape derivative. Thus,   letting  $\ell:=  \HH ^ 1 ( [A, B])$, it holds
\begin{equation}\label{f:sder} 
 \frac{d}{d\e}\Big|_{\e=0}\T_{\rm mix} (T_\e) 
= \int_{A}^{B}\left(\frac{\ell}{2}-|x-A|\right)\,|\nabla w_T|^2\,d\HH^1(x)\,.
 \end{equation}

Notice that the boundary integral above is well defined. Indeed, by the standard corner expansion for mixed Dirichlet--Neumann problems, see e.g. \cite{Dauge88}, at each (DN)-vertex of opening $\beta<\pi$ the first singular exponent is $\frac{\pi}{2\beta}>\frac12$, and therefore $|\nabla w_T|^2$ is integrable on $\partial T$.

As a consequence of \eqref{f:sder}, condition 
\eqref{f:statmix} is equivalent to requiring  that $w_T$ satisfy the following
overdetermined boundary condition,   where $M$ is the midpoint of $[A, B]$:
 \begin{equation}\label{f:critical}
\int_{A}^{M}|x-M|\,|\nabla w_T|^2\,d\HH^1(x)= \int_{M}^{B}|x-M|\,|\nabla  w_T|^2\,d\HH^1(x)\,.
\end{equation}

The next symmetry statement is at the heart of our proof. It  may be viewed as an 
analogue, for triangles, of the symmetry result in sector-like domains  established in 
\cite[Theorem 1.1]{PaTr21}.

 \begin{proposition}\label{p:reflection} 
  Let $\alpha \in (0, \frac{\pi}{2}]$. Then 
 $T\in\mathcal A _\alpha $ is  critical for $\mathcal E _{mix}$ under rotation of its Dirichlet side
 $\Gamma$ in the sense of  \eqref{f:statmix} 
  if and only if it is isosceles on the base $\Gamma$. 
Equivalently,  if   
   $T= \triangle(O, A, B)$, with $\Gamma= [A, B]$, the mixed torsion function $w _T$  satisfies  the overdetermined condition  \eqref{f:critical}  if and only if $OA = OB$. 
     \end{proposition} 
     
The proof of Proposition \ref{p:reflection}  relies on the following lemma.

     \begin{lemma}\label{l:angular} 
     Under the same assumptions as in Proposition \ref{p:reflection},  the angular derivative of $w_T$
    in a system of coordinates having vertex in $A$ or in $B$, and horizontal axis containing $[A, B]$,  is nonnegative in $T$.     \end{lemma} 
     
 \begin{proof} For definiteness, let us 
 consider a system of coordinates with $A= (0, 0)$, and 
let us assume with no loss of generality that 
  $$B = (B _ x, 0) \ \text{ with } B _ x >0 \, , \qquad O = (O_x, O_y)  \ \text{ with }  O_y >0\,.$$  
 The angular derivative of $w _T$ is given by 
 $$  w _ \theta   := (-y,x)\cdot \nabla w_T\,.
             $$ 
Away from the vertices, $w_\theta$ is smooth. 
Moreover, since $
w_\theta=(-y\partial_x+x\partial_y)w_T$, 
and $-\Delta w_T=1$, a direct computation gives  
\begin{equation}\label{f:equtheta}
\Delta w_\theta=0 \qquad \text{in }T\,.
\end{equation}
 We claim that 
 \begin{equation}\label{f:bordoutheta} 
 w_\theta \geq  0 \qquad \text { on } \partial T \,. 
 \end{equation} 
  Indeed, on the Dirichlet side $[A,B]$ we have
$w_\theta=x\,\partial_y w_T\geq0$, since $x\geq0$, $w_T$ vanishes
on $[A,B]$, and it is positive in $T$.  
 Moreover, we see immediately that  $w_\theta = 0 $ on the Neumann side $[A, O]$,  because on this side 
  $(-y,x)$ is parallel to   the outer
normal to $T$.

To conclude the proof of \eqref{f:bordoutheta}, we have to show that
\begin{equation}\label{f:angularsign}
w_\theta\geq0\qquad\text{on }[B,O]\,.
\end{equation}
Let $\tau_{[B,O]}$ denote the unit tangent vector to $[B,O]$, oriented from $B$ to $O$. Since the angle at the (NN)-vertex $O$ is not larger than $\pi/2$, by \cite[Theorem 1.3]{LY} the function $w_T$ is nondecreasing in the direction of the inward normal to the Dirichlet side $[A,B]$, which in our coordinates is the vertical direction. Hence $\partial_y w_T\geq0$. On the other hand, the Neumann condition on $[B,O]$ implies that, away from the vertices, for some $c \in \R$,
$$
\nabla w_T=c\,\tau_{[B,O]}\,.
$$
Since
$$
\tau_{[B,O]}\cdot(0,1)=\frac{O_y}{|O-B|}>0,
$$
it follows that $c\geq0$.

It remains to observe that $(-y,x)\cdot\tau_{[B,O]}>0$ along $[B,O]$. Indeed, writing $B=(B_x,0)$ and $X=B+t(O-B)$, with $t\in[0,1]$, a direct computation gives
$$
(-X_y,X_x)\cdot\tau_{[B,O]}
=\frac{B_xO_y}{|O-B|}>0.
$$
Therefore,
$$
w_\theta=(-y,x)\cdot\nabla w_T
=c\,(-y,x)\cdot\tau_{[B,O]}\geq0
\qquad\text{on }[B,O]\,.
$$

It remains to propagate the boundary inequality to the interior.
Unlike in the proof of Theorem~\ref{t:strategy-asymptotics}, here the
(DN)-vertices may be obtuse, so that $w_T$ need not belong to $H^2(T)$
and consequently $w_\theta$ need not belong to $H^1(T)$. We therefore
cannot directly use the negative part of $w_\theta$ as a test function,
and argue instead on a truncated domain.  
 
For $i = 1, 2$, let $P_i$   be the (DN)-vertices of $T$, with opening angles $\beta_i$.
Since $\beta_i<\pi$, we can choose $q_i$ such that 
$$
\max\{0,1-\frac{\pi}{2\beta_i}\} < 
q_i<\frac{\pi}{\beta_i}.
$$
In polar coordinates  centred at $P_i$, with
$0<\theta<\beta_i$, the function
$$
\psi _i (r,\theta)
:=
r^{-q_i}
\cos\Big ( q_i\big(\theta-\frac{\beta_i}{2}\big) \Big)
$$
is positive (thanks to the upper inequality satisfied by $q_i$) and harmonic in the corresponding sector.

Let $\Psi:= \psi _ 1+ \psi _2$ be the sum of these barriers over the two
(DN)-vertices. We claim that, for every $\delta>0$ and for all sufficiently small
$\varepsilon>0$,
$$
w_\theta+\delta\Psi\geq0
$$
on the boundary of the domain obtained from $T$ by removing
$\varepsilon$-neighbourhoods of  $P_1$ and $P_2$. 

By the standard corner expansion for mixed
Dirichlet--Neumann problems, see e.g. \cite{Dauge88},
if $\beta$ is the opening of a (DN)-vertex, the first
singular exponent is
$
\lambda=\frac{\pi}{2\beta}.
$
 Since the vector field $(-y,x)$ is bounded on $T$, it follows in particular that, near $P_i$, 
$$
|w_\theta(x)|
\leq C\left(1+|x-P_i|^{\frac{\pi}{2\beta_i} -1}\right).
$$ 
On the other hand, on the circular arc $r=\varepsilon$ we have
\[
\psi_i \geq c_i\varepsilon^{-q_i},
\qquad
c_i:=\cos\left(\frac{q_i \beta_i}{2}\right)>0.
\]
Hence, 
\[
w_\theta+\delta\psi_i
\geq
-C\left(1+\varepsilon^{\frac{\pi}{2\beta_i} -1}\right)
+\delta c_i\varepsilon^{-q_i}.
\]
Since 
$q_i>\max\{0,1-\frac{\pi}{2\beta_i} \}$, 
the right-hand side of the above inequality is nonnegative for all sufficiently small
$\varepsilon>0$. Since $\Psi\geq\psi_i$, the claim follows on the
circular arcs. On the remaining part of the boundary, it follows from
\eqref{f:bordoutheta} and the positivity of $\Psi$. 

 Since both $w_\theta$ and $\Psi$ are harmonic in the truncated domain,
the maximum principle yields
$w_\theta+\delta\Psi\geq0$ there.
Since $\varepsilon$ can be taken arbitrarily small, we obtain
$w_\theta+\delta\Psi\geq0$ away from the (DN)-vertices.
Letting $\delta\to0$, we conclude that
$w_\theta\geq0$ in $T$. 
\end{proof}

 \bigskip \medskip

\begin{proof}[Proof of Proposition \ref{p:reflection}]
When $T$ is isosceles,  condition  \eqref{f:critical} is fulfilled because
$w_T$ is symmetric with respect to the axis of symmetry of $T$.
Conversely, we have to show 
that, if $T$ is not isosceles, 
$w _T$  does not satisfy  \eqref{f:critical}.
With no loss of generality, let  $\angle OAB $ be strictly larger than $\angle OBA$.  
Consider the perpendicular  to $[A, B]$ through $M$, and denote by $N$ its intersection with  $[A, O]$. 
Consider the triangles
$$ \Om:= \triangle (B,N,M) \qquad \text{ and } \qquad \Om ^* = \triangle(N,M,A) \,, $$ 
see Figure \ref{fig:4}.

      \begin{figure} [h] 
     \includegraphics[height=4cm]{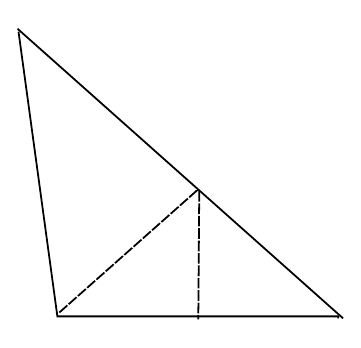}%
         \put(-120, 105){\color[rgb]{0,0,0}\makebox(0,0)[lb]{\smash{$O$}}}%
         \put(-55, 55){\color[rgb]{0,0,0}\makebox(0,0)[lb]{\smash{$N$}}}%
  \put(-110, 0 ){\color[rgb]{0,0,0}\makebox(0,0)[lb]{\smash{$A$}}}%
  \put(-60, 0 ){\color[rgb]{0,0,0}\makebox(0,0)[lb]{\smash{$M$}}} 
  \put(-10, 0 ){\color[rgb]{0,0,0}\makebox(0,0)[lb]{\smash{$B$}}} 
\caption{The geometry of Proposition \ref{p:reflection}  }
\label{fig:4}   
\end{figure}
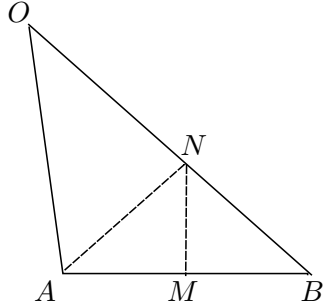

For any $x \in \Om ^*$,   denote by $x^*$ its reflection  with respect to $[N, M]$, and 
define 
$$z (x):= w_T  (x) - w_T (x^*) \qquad\forall x \in \Om ^*\,.$$  
If $\nu$ is  the normal to $[A,N]$ in direction outward to  $\Om ^*$, by Lemma \ref{l:angular} we have 
$$ \partial _\nu z \geq 0 \qquad \text{ on }  [A, N]\,.$$ 
 Then $z$ satisfies the following boundary value problem in $\Om^*$:
$$
\begin{cases}
\Delta z=0 & \text{ in }\Om^*,\\
z=0 & \text{ on }[N,M]\cup[A,M],\\
\partial_\nu z\geq0 & \text{ on }[A,N].
\end{cases}
$$
We claim that this implies 
\begin{equation}\label{f:positivityw} 
z  \geq 0 \qquad \text{  in  } \Om^*\,.
\end{equation} 
Indeed, for every nonnegative function $v \in H^1(\Om^*)$ having zero trace on $[N,M]\cup[A,M]$, we have
$$
0\leq\int_{[A,N]}v \,\partial_\nu z =\int_{\Om^*}{\rm div}(v\nabla z)=\int_{\Om^*}\nabla z\cdot\nabla v\,.
$$
Choosing $v=z^-:=-\min\{z,0\}$, we get
$$
0\leq-\int_{\{z<0\}}|\nabla z|^2\,,
$$
which implies \eqref{f:positivityw}.

 Moreover, $z\not\equiv0$ in $\Om^*$, since otherwise $w_T$
would be symmetric with respect to $[N,M]$, which would force
$T$ to be isosceles, contrary to our assumption.   
Then, by the Hopf boundary point lemma,
 denoting by $\nu$ 
 the normal to $[A, M]$ in direction outward to  $\Om^*$, we have the strict inequality
 $\partial_{\nu }  z <0$ on $[A, M]$, yielding that  \eqref{f:critical} 
 is not satisfied. \end{proof}
 
\medskip 
 
\section{Proof of Theorem \ref{t:concavity}} \label{sec:concavity}

Recall that $F$ is defined by \eqref{f:effe}, where  
 $\mathcal T (\alpha)$ is the mixed torsional rigidity  of 
the isosceles triangle having equal sides $1$ and vertex angle $\alpha$ (with  Dirichlet side given by  the base opposite to the vertex  angle $\alpha$). 
Notice that, letting  $s :=\frac\alpha2$,  by symmetry, $\mathcal T (\alpha)$ equals twice  the mixed  torsional rigidity  of the triangle with  vertices $(0,0),  (\cos s, 0), (\cos s, \sin s)$ (with Dirichlet condition  on the vertical side).

\medskip 

 \underbar{Step 1} (Pulling back $F$ to a fixed domain).
We let 
$$D= \{(x,y) : 0<x<1, 0<y<x\} \,,$$
and we denote by $\Gamma$ its vertical side. 
 After rescaling by a factor $\frac{1}{\cos s}$ and applying the affine transformation $(x,y) \to (x, \frac{y}{\tan s})$, we  can write
$F (\alpha)$ 
as 
\begin{equation}\label{f:repF}
F(\alpha)= \frac{2s^2}{\tan s} \int_D v_{t}= : G(s).
\end{equation}  
Here,   $t= \frac 1{\tan^2 s}$,  and $v_t$ denotes  the unique solution to  the following problem: 
\begin{equation}\label{f:defvt} v_t \in H^1_\Gamma (D) 
\ :\quad  \mathcal B_t(v_t, \vphi)= \int_D\vphi \quad \forall \vphi \in H^1_\Gamma (D)\,,
\end{equation} 
where 
$\mathcal B_t: H^1_\Gamma(D) \times H^1_\Gamma (D) \to \R$ the coercive and continuous bilinear form defined by
$$\mathcal B_t(u,v):= \int_D \frac{\partial u}{\partial x} \frac{\partial v}{\partial x}+ t \frac{\partial u}{\partial y} \frac{\partial v}{\partial y} dx dy\,.$$  
In view of the representation formula \eqref{f:repF},
our objective is to prove that 
$$ G''(s) <0 \qquad \forall s \in  (0,    \frac {\pi}{2} )\,. $$
We point out that, although  also the variable $t$ appears on  the right--hand side  of \eqref{f:repF}, it is directly related to $s$ by the equality $t= \frac 1{\tan^2 s}$. The reason why  we prefer to retain the notation $t$  is that, in the sequel,  it will make
a series of simplifications  more transparent,  as a consequence of the affine dependence of $\mathcal B_t$ 
 on $t$.

\medskip
 \underbar{Step 2}   (Decomposition of   the function $v_t$).
 Adopting the same notation as in  \cite[Section 3]{GHL26},  we set  
$$p(x):=\frac{1-x^2}4\, , \qquad e(x,y):= \frac{1-x^2}4+y(x-1)\,, $$
 and we let $X,Y:H^1_\Gamma(D)\times H^1_\Gamma(D)\to\mathbb R$ be the bilinear forms defined by
$$
X(u,v):=\int_D \frac{\partial u}{\partial x}\frac{\partial v}{\partial x},
\qquad
Y(u,v):=\int_D \frac{\partial u}{\partial y}\frac{\partial v}{\partial y},
$$
so that
$$
\mathcal B_t=X+tY.
$$
We claim that the function $v_t$ defined in 
\eqref{f:defvt}  can be recast as 
\begin{equation}\label{bbfc01bis}
v_t= q_t+p
\end{equation}
 where 
$q_t$ is the unique solution to 
\begin{equation}\label{f:defqt} q_t \in H^1_\Gamma (D) 
\ :\quad  \mathcal B_t(q_t, \vphi)= X(e, \vphi)  \quad \forall \vphi \in H^1_\Gamma (D)\,.
\end{equation} 

Indeed, we have 
\begin{equation}\label{bbfc01.1} \mathcal B_t(v_t, \vphi)- \mathcal B_t(q_t, \vphi)= \int_D \vphi-X(e, \vphi)=\int_D \vphi-\int_D (-\frac x2+y)\frac{\partial \vphi}{\partial x}
\end{equation}
and
\begin{equation}\label{bbfc01.2}\mathcal B_t(p, \vphi)= \int _D\frac{\partial p}{\partial x} \frac{\partial \vphi}{\partial x} =-\int _D\frac{x}{  2} \frac{\partial \vphi}{\partial x}.
\end{equation}
Then our claim \eqref{bbfc01bis} follows from the equality between the right--hand sides of \eqref{bbfc01.1} and \eqref{bbfc01.2}, 
which can be checked via a simple integration by parts. 
 
Now, in view of \eqref{bbfc01bis}, taking into account that $p$ is independent of $t$
and that $q_t$ depends on $s$ through
$
t=\frac{1}{\tan^2 s},
$
the chain rule shows that the computation of $G''(s)$ only requires
$$\int_D q_t, \, \qquad  \frac{d}{dt} \int_D q_t \, , \qquad  \frac{d^2}{dt^2} \int_D q_t.$$
More precisely, if we set for convenience 
$$M_1(t)= 48\int_D q_t, $$
recalling that  $t= \frac{1}{\tan ^2s}$, after a suitable rearrangement of the terms, we can write
\begin{equation}\label{f:Gsec} 
G''(s)= A_0(s)\Big(\frac 1{16}+ \frac{M_1(t)}{48}\Big)+ A_1(s) \frac{M_1'(t)}{48}+ A_2(s)  \frac{M_1''(t)}{48},
\end{equation} 
with
$$
\begin{aligned}
& A_0(s)=  \frac{4}{\tan s} \Big(\frac{s^2}{\sin ^2 s} -\frac{2s}{\sin s \cos s} +1\Big),
\\ 
& A_1(s)= \frac{4 s }{\tan s \sin^2 s} \Big(\frac{5s}{\tan ^2 s}+3s-\frac{4}{\tan s}\Big),
\\ 
& A_2(s)= \frac{8s^2}{\tan ^3 s\sin^4 s}.
\end{aligned} 
$$ 
Now, the first and second order derivatives $M _ 1' ( t) $ and $M_1'' (t)$ appearing in \eqref{f:Gsec}  
can be expressed in terms of 
 $$M_2(t):= 48 X(q_t, q_t),  \qquad   M_3(t) := 48 X(q_t, r_t)\,,$$ 
 where  $r_t$ is the unique solution to 
 \begin{equation}\label{bbfc00.2}
r_t \in H^1_\Gamma (D) 
\ :\quad  \mathcal B_t(r_t, \vphi)= X(q_t, \vphi)  \quad \forall \vphi \in H^1_\Gamma (D)\,. 
\end{equation} 
Indeed,  differentiating \eqref{f:defqt}  with respect to $t$, we obtain
$$
q_t' \in H^1_\Gamma(D) : \qquad
\mathcal B_t(q_t',\varphi)=-Y(q_t,\varphi)
\qquad \forall \varphi\in H^1_\Gamma(D).
$$  
On the other hand,
$$
\mathcal B_t\!\left(\frac{r_t-q_t}{t},\varphi\right)
=
\frac{1}{t}\Bigl(X(q_t,\varphi)-\mathcal B_t(q_t,\varphi)\Bigr)
=
-Y(q_t,\varphi),
$$
for every $\varphi\in H^1_\Gamma(D)$. Hence, by uniqueness,
$$
q_t'=\frac{r_t-q_t}{t}.
$$
Using this identity, together with the equations satisfied by $v_t$, $q_t$ and $r_t$, we obtain
$$
M_1'(t)
=
-\frac{M_1(t)-M_2(t)}{t},
\qquad
M_1''(t)
=
\frac{2}{t^2}
\Bigl(M_1(t)-2M_2(t)+M_3(t)\Bigr).
$$

\medskip
 \underbar{Step 3} 
(Elementary estimates for $M _ 1 ( t)$, $M _ 1' ( t)$ and $M_ 1'' (t)$. We have 
\begin{equation}\label{bbfc02} \frac{1}{1+4t}< M_1(t) \le \frac{1}{1+t} <1,\end{equation}
\begin{equation}\label{bbfc03}0\le M_3(t) \le M_2(t),\end{equation}
\begin{equation}\label{bbfc04}M_1''(t) \le -\frac 2t M_1'(t) \quad \mbox{ and } \quad M_1'(t) \ge -\frac{M_1(t)(1-M_1(t))}{t}.\end{equation}
To prove \eqref{bbfc02} we first observe that $X(e,p)=0, X(e,e)= \frac 1{48}$ and that
$$\int_Dq_t= \mathcal B_t(v_t, q_t)= \mathcal B_t(q_t+p, q_t)= X(e, q_t)+X(e,p)= X(e, q_t)=-\int_D \frac y2 \frac{\partial q_t}{\partial y}.$$
Hence
$$\Big (\int_Dq_t\Big )^2 =   X(e, q_t) ^2 \le   X(e,e) X(q_t,q_t)= \frac{X(q_t,q_t)}{48}\,,$$
or 
$$M_1(t)^2 \le M_2(t).$$
Similarly, we have 
$$\Big (\int_Dq_t\Big )^2 =   \Big (\int_D \frac y2 \frac{\partial q_t}{\partial y}\Big )^2 \le   \int_D \frac {y^2}{4} \int_D \Big ( \frac{\partial q_t}{\partial y}\Big )^2=\frac{Y(q_t,q_t)}{48}, $$
or
$$M_1(t)^2 \le 48 Y(q_t,q_t).$$
But
$M_1(t)= M_2(t)+ t 48 Y(q_t,q_t)$, hence $M_1(t) \ge (1+t) M_1^2(t)$, and the  right-hand inequality in \eqref{bbfc02} follows.
For the left-hand inequality, we have
\[
 \left(\frac1{48}\right)^2
 = X(e,e)   ^ 2  = \mathcal B_t(q_t,e)^2
 \leq \mathcal B_t(q_t,q_t)\mathcal B_t(e,e)
 =\frac{M_1(t)}{48}\frac{1+4t}{48}.
\]
 Above, we used the elementary evaluation $\mathcal B_t(e,e)= \frac 1{48} + \frac{t}{12}$.

To prove \eqref{bbfc03}, we apply the Cauchy-Schwarz inequality
$$X(r_t,r_t) \le \mathcal B_t(r_t,r_t)= X(q_t,r_t)  \le X(q_t, q_t)^\frac 12 X(r_t, r_t)^\frac 12.$$
Finally, using the expressions of $M_1'(t)$ and $M_1''(t)$ obtained in Step 2, together with \eqref{bbfc03}, we infer  
$$
M_1''(t)
=
\frac{2}{t^2}\bigl(M_1(t)-2M_2(t)+M_3(t)\bigr)
\le
\frac{2}{t^2}\bigl(M_1(t)-M_2(t)\bigr)
=
-\frac{2}{t}M_1'(t).
$$
Moreover, since $M_1(t)^2\le M_2(t)$, we have
$$
M_1'(t)
=
-\frac{M_1(t)-M_2(t)}{t}
\ge
-\frac{M_1(t)-M_1(t)^2}{t}
=
-\frac{M_1(t)(1-M_1(t))}{t}.
$$
This proves 
\eqref{bbfc04}.

\medskip 

 \underbar{Step 4} (Estimate of $G''$).
We first prove that $G''$ satisfies the upper bound
\begin{equation}\label{bbfc05} 
G''(s) \le \frac{1}{12 \tan s} \Phi\big (M_1(\frac{1}{\tan ^2 s}), s\big)  \qquad \forall s \in (0, \frac \pi2)  
  \end{equation}
where 
$$\Phi(M, s):= \frac s{\tan s} R(s) M(1-M)-Q(s)(M+3),$$
with
$$
\begin{aligned}
& R(s)= \frac{s}{\tan s} (\tan^4 s-1) +\frac{4}{\cos^2 s} >0\qquad \forall s \in (0, \frac \pi2)
,
\\ 
& Q(s)=  \frac{1}{\sin ^2 s} \Big( s(\tan s -s)+ (s\tan s-\sin^2 s)\Big )>0 \qquad \forall s \in (0, \frac \pi2).
\end{aligned} 
$$ 
Indeed, since 
$$A_0(s)= -\frac {4Q(s)}{\tan s} \quad \text{ and } \quad  2A_2(s)\tan ^2 s -   A_1(s)= \frac{4sR}{\tan ^4 s} >0.$$
Using the inequalities 
\eqref{bbfc04} established in Step 3, we obtain  
$$\begin{aligned} 
 A_1(s) \frac{M_1'(t)}{48}+ A_2(s)  \frac{M_1''(t)}{48}
 &
 \le \Big (A_1(s)- 2 A_2(s) \tan^2 s\Big ) \frac{M_1'(t)}{48}
 \\
&  \le \frac{s}{12 \tan ^2 s}R(s) M_1(t)(1-M_1(t))\,. 
\end{aligned} $$
Starting from the expression  \eqref{f:Gsec} of $G''$  and using the above estimate, we obtain \eqref{bbfc05}.

 Now, in view of \eqref{bbfc05}, it remains to show that
\begin{equation} \label{f:fino} 
 \Phi\big (M_1(\frac{1}{\tan ^2 s}), s\big)  <0  \qquad \forall s \in (0, \frac{\pi}{2}) \,.\
 \end{equation}
To this end,  we use the fact that, for each fixed $s$, the map $M\mapsto\Phi(M,s)$ is quadratic and concave.  Notice that,  by  \eqref{bbfc02}, $M_1(t)$ belongs to the interval 
$\bigl[\frac{\sin^2 s}{1+3\cos^2 s}, \sin^2 s\bigr]$ (with $t=1/\tan^2 s$). 

We divide the proof  of \eqref{f:fino} into three cases, corresponding to three ranges of $s$ that together cover $(0,\pi/2)$. 
\medskip

\medskip
\noindent {\it Range I: $0<\tan  s \leq2/3$.}
Since the function $\Phi$ is quadratic and concave with respect to the variable $M$, taking also into account that $M_1(t) \le \sin ^2s$, it is enough   to prove the two inequalities: 
$$ \Phi (\sin^2 s, s)<0 \qquad \text{ and } \qquad \frac{\partial \Phi}{\partial M} (\sin ^2 s,s) \ge 0\,.$$

Concerning $ \Phi (\sin^2 s, s)$, its explicit expression reads: 
$$\begin{aligned}\label{eq:Phiu}
 & -\frac{3\tan^6s+2(\tan s -s)\tan^3s(1+\tan^2s)-(\tan s -s)^2(5\tan^4s+6\tan^2s+3)}
 {\tan^2s(1+\tan^2s)} 
 \\
 & < -\frac{20}{9}\frac{\tan^4s}{1+\tan^2s}<0 \,,
 \\ 
 \end{aligned}$$
 where we have used the elementary inequalities 
$5(\frac23)^4+6(\frac23)^2+3\leq539/81<7$ and
$ 0<\tan s -s \leq {\tan^3s}/{3}$.
 
 Concerning $ \frac{\partial \Phi}{\partial M} (\sin ^2 s,s)$, its explicit expression reads:  
$$\begin{aligned} 
 &3-3\tan^2s-\tan^4s+(2\tan^3s-2/\tan s)(\tan s -s)+(3-\tan^2s)(\tan s -s)^2\\
 &\geq 3-\frac{11}{3}\tan^2s-\tan^4s+\frac23\tan^6s
 \geq 3-\frac{44}{27}-\frac{16}{81}
 =\frac{95}{81}>0\,, 
\end{aligned}
$$ where we have used the elementary inequalities   $2\tan^3s-2/\tan s<0$, $\tan s-s\leq \tan^3s/3$, and
$3-\tan^2s>0$. 

\bigskip 
\noindent{\it Range II: $2/3\leq \tan s\leq2$.}
Since for every $0\leq M\leq1$ we  have 
$$\Phi(M, s)= \frac s{\tan s} R(s) M(1-M)-Q(s)(M+3)\le  \frac s{4\tan s} R(s) -3Q(s) \,, $$
it is enough to prove the inequality  
\begin{equation}\label{f:enough} 
12Q(s)-\frac s{\tan s}R(s) >0\,.
\end{equation}  
For a given $s$, let  $a  \mapsto P ( a) $ be the concave quadratic function defined by 
$$
 P(a):=-(\tan ^2s+12+11\tan^{-2}s)a^2+20(\tan s+\tan^{-1}s)a-12\,,
$$ 
so that the left-hand side of \eqref{f:enough} is precisely $P (s)$. 
 
By the integral representation of the arctangent, the ratio $s/\tan s$ can be written as the average over $(0,1)$ of the function $w\mapsto(1+w)^{-1}$ evaluated at $w=v^2(\tan s)^2$. Since this function is convex, Jensen's inequality allows us to replace $v^2(\tan s)^2$ by its average over $(0,1)$, namely $(\tan s)^2/3$. So we have: 
$$
 \frac{s}{\tan s}=\int_0^1\frac{d v}{1+\tan ^2sv^2}
 \geq\frac{1}{1+\frac 13\tan^2s }
$$
Together with the elementary inequality $s\leq\tan s$, this shows that 
 $s$ lies in the interval $[\frac{3\tan s}{\tan^2s+3},  \tan s]$.   Since $P$ is concave, its minimum on this interval is attained
at one of the endpoints. Hence it is enough to check the two
endpoint values:   
 
\[
 P(\tan s)=-\tan^4s+8\tan^2s-3,\qquad
 P\left(\frac{3\tan s}{\tan^2s+3}\right)
 =\frac{3(13\tan^4s+20\tan ^2s-9)}{(\tan ^2s+3)^2}.
\]
 Both are strictly positive when
$4/9\leq \tan^2s\leq4$. 
 
\bigskip 
\noindent{\it Range III: $\tan s \geq2$.}
Set $$ \ell(s):=\frac{\tan^2 s}{\tan^2 s+4}. $$ Since $\tan s\geq2$, we have $\ell(s)\geq1/2$. Moreover, by \eqref{bbfc02}, with $t=1/\tan^2s$, we have $$ M_1(t)\geq \ell(s). $$ Since the map $M\mapsto M(1-M)$ is decreasing on $[1/2,1]$, and since $M_1(t)\in[\ell(s),1]$, we obtain $$ M_1(t)(1-M_1(t)) \leq \ell(s)(1-\ell(s)). $$ At the same time, since $Q(s)>0$ and $M_1(t)\geq\ell(s)$, $$ -Q(s)(M_1(t)+3) \leq -Q(s)(\ell(s)+3). $$ Therefore,
\begin{equation}\label{eq:largePhi} 
 \Phi(M_1 ( t),s)
 \leq \frac s{\tan s}R(s)\ell(s)(1-\ell(s))-Q(s)(\ell(s)+3) 
 =-\frac{4}{(\tan^2s+4)^2}   H(s)     \end{equation}
where we have set
\begin{align*}
 H(s):=&s\left[2\tan^5s+12\tan ^3s+34\tan s+\frac{24}{\tan s}
       -s\left(2\tan ^4s+8\tan ^2s+18+\frac{12}{\tan ^2s}\right)\right]\\
 &\hspace{30mm}-(\tan ^4s+7\tan ^2s+12).
\end{align*}
Thus it remains to show that $H$ is positive in the range under consideration.
For $\tan s\geq2$, we have the simple bounds
\begin{equation}\label{eq:slarge}
 1<s<\frac{2\tan s}{3}.
\end{equation}
Using this upper bound inside the square bracket appearing in the definition of $H$,
we see that the bracket exceeds the positive quantity 
\[
 \frac23\tan ^5s+\frac{20}{3}\tan ^3s+22\tan s+\frac{16}{\tan s}.
\]
We may consequently use $s>1$ outside the brackets, obtaining
\begin{align*}
 H(s)
 &>\frac23\tan ^5s-\tan ^4s+\frac{20}{3}\tan ^3s-7\tan ^2s+22\tan s-12+\frac{16}{\tan s}\\
 &=\tan ^4s\left(\frac{2\tan s}{3}-1\right)
   +\tan ^2s\left(\frac{20\tan s}{3}-7\right)+22 \tan s  -12+\frac{16}{\tan s}>0.
\end{align*}
Every term on the last line is positive when  $\tan s \geq2$.
We conclude that $\Phi(M_1(\frac 1{\tan^2s}),s)<0$ also in the
third range.

\medskip
Combining the three ranges, we  conclude that $G''(s)<0$ for
$0<s<\pi/2$.  
 \qed

\bigskip

 \section*{Statement on AI Assistance.} 
 
The proof of Theorem~\ref{t:concavity} was developed with the assistance of ChatGPT 6 Astra, in particular concerning the trigonometric estimates,  with repeated checks and reworkings by the authors throughout the process.

\section*{
Conflict of Interest and Data Availability.} The authors declare that they have no conflicts of interest.
No datasets were generated or analyzed during the current study.

\end{document}